\documentclass{amsart}

\usepackage{amssymb,amsthm,url,tikz}
\usetikzlibrary{arrows,shapes,positioning}

\newtheorem{theorem}{Theorem}[section]
\newtheorem{lemma}[theorem]{Lemma}
\newtheorem{corollary}[theorem]{Corollary}

\theoremstyle{definition}
\newtheorem{definition}[theorem]{Definition}

\newcommand{\Ker}{\mathop{\textrm{Ker}}}

\newcommand{\Aut}{\mathop{\mathrm{Aut}}}

\begin{document}

\title[Intransitive graph-restrictive groups]{On intransitive graph-restrictive permutation groups}

\author[P. Spiga]{Pablo Spiga}
\address{Pablo Spiga, Departimento di Matematica Pura e Applicata,\newline
 University of Milano-Bicocca, Via Cozzi 53, 20126 Milano, Italy} 
\email{pablo.spiga@unimib.it}

\author[G. Verret]{Gabriel Verret}
\address{Gabriel Verret, Faculty of Mathematics, Natural Sciences and Info. Tech.,  \newline
University of Primorska, Glagolja\v{s}ka 8, 6000 Koper, Slovenia}
\email{gabriel.verret@fmf.uni-lj.si}

\thanks{Address correspondence to G. Verret (gabriel.verret@fmf.uni-lj.si)}

\subjclass[2010]{Primary 20B25; Secondary 05E18}

\keywords{vertex-transitive, graph-restrictive, semiregular}

\begin{abstract}
Let $\Gamma$ be a finite connected $G$-vertex-transitive graph and let $v$ be a vertex of $\Gamma$. If  the permutation group induced by the action of the vertex-stabiliser $G_v$ on the neighbourhood $\Gamma(v)$ is permutation isomorphic to $L$, then $(\Gamma,G)$ is said to be \emph{locally-$L$}. A permutation group $L$ is \emph{graph-restrictive} if there exists a constant $c(L)$ such that, for every locally-$L$ pair $(\Gamma,G)$ and a vertex $v$ of $\Gamma$, the inequality $|G_v|\leq c(L)$ holds. We show that an intransitive group is graph-restrictive if and only if it is semiregular.
\end{abstract}

\maketitle

\section{Introduction}
A graph $\Gamma$ is said to be $G$-\emph{vertex-transitive} if $G$ is a subgroup of $\Aut(\Gamma)$ acting transitively on the vertex-set of $\Gamma$. Let $\Gamma$ be a finite, connected, simple $G$-vertex-transitive graph and let $v$ be a vertex of $\Gamma$. If  the permutation group induced by the action of the vertex-stabiliser $G_v$ on the neighbourhood $\Gamma(v)$ is permutation isomorphic to $L$, then $(\Gamma,G)$ is said to be \emph{locally-$L$}. Note that, up to permutation isomorphism, $L$ does not depend on the choice of $v$, and,  moreover, the degree of $L$ is equal to the valency of $\Gamma$. In~\cite[page~$499$]{Verret}, the second author introduced the following definition.

\begin{definition}\label{def:Restrictive} 
A permutation group $L$ is \emph{graph-restrictive} if there exists a constant $c(L)$ such that, for every locally-$L$ pair $(\Gamma,G)$ and for every vertex $v$ of $\Gamma$, the inequality $|G_v|\leq c(L)$ holds.
\end{definition} 

To be precise, Definition~\ref{def:Restrictive} is a generalisation of the definition from~\cite{Verret}, where the group $L$ is assumed to be transitive. The problem of determining which transitive permutation groups are graph-restrictive was also proposed in~\cite{Verret}. A survey of the state of this problem can be found in~\cite{PSVRestrictive}, where it was conjectured (\cite[Conjecture~3]{PSVRestrictive}) that a transitive permutation group is graph-restrictive if and only if it is semiprimitive. (A permutation group is said to be \emph{semiregular} if each of its point-stabilisers is trivial and \emph{semiprimitive} if each of its normal subgroups is either transitive or semiregular.)

Having removed the requirement of transitivity from the definition of graph-restrictive, it is then natural to try to determine which intransitive permutation groups are graph-restrictive. The main result of this note is a complete solution to this problem (which we did not expect, given the abundance and relative lack of structure of intransitive groups).

\begin{theorem}\label{theo:main}
An intransitive and graph-restrictive permutation group is semiregular.
\end{theorem}

It is easily seen that a semiregular permutation group is graph-restrictive. Indeed, if $L$ is a semiregular permutation group of degree $d$ and $(\Gamma,G)$ is locally-$L$, then for every arc $vw$ of $\Gamma$ the group $G_{vw}$ fixes the neighbourhood $\Gamma(v)$ pointwise. Since $\Gamma$ is connected, it follows that $G_{vw}=1$ and hence $|G_v|\leq |\Gamma(v)|=d$ and $L$ is graph-restrictive. Thus Theorem~\ref{theo:main} provides a characterisation of intransitive graph-restrictive groups.

\begin{corollary}\label{cor:main}
An intransitive permutation group is graph-restrictive if and only if it is semiregular.
\end{corollary}

Note that an intransitive permutation group is semiregular if and only if it is semiprimitive. In particular, Corollary~\ref{cor:main} completely settles the intransitive version of~\cite[Conjecture~3]{PSVRestrictive}, giving remarkable new evidence towards its veracity.

\section{Proof of Theorem~\ref{theo:main}}
For the remainder of this paper, let $L$ be a permutation group on a finite set $\Omega$ which is neither transitive nor semiregular. We show that $L$ is not graph-restrictive, from which Theorem~\ref{theo:main} follows.

\subsection{The construction}\label{sub:idiot}
 Let $\omega_1,\ldots,\omega_k\in\Omega$ be a set of representatives of the orbits of $L$ on $\Omega$. Since $L$ is not transitive, $k\geq 2$ and, since $L$ is not semiregular, we may assume without loss of generality that $L_{\omega_1}\neq 1$. Let $n\geq 2$ be an integer and  let $b_1$ be the automorphism of $L_{\omega_1}\times L_{\omega_1}^n=L_{\omega_1}^{n+1}$ defined by
$$(x_0,x_1,\ldots,x_{n-1},x_n)^{b_1}=(x_n,x_{n-1},\ldots,x_1,x_0),$$
for each $(x_0,\ldots,x_{n})\in L_{\omega_1}^{n+1}$. Similarly, let $b_2$ be the automorphism of $L_{\omega_1}^{n}$ defined by 
$$(x_1,x_2,\ldots,x_{n-1},x_n)^{b_2}=(x_n,x_{n-1},\ldots,x_2,x_1),$$
for each $(x_1,\ldots,x_n)\in L_{\omega_1}^n$. Clearly, $b_1$ and $b_2$ are involutions, that is, $b_1^2=1$ and $b_2^2=1$. Now, let $\langle b_3\rangle,\ldots,\langle b_k\rangle$ be cyclic groups of order $2$ and consider the following abstract groups:  
\begin{eqnarray*}
A&:=& L\times L_{\omega_1}^n,\\
B_1&:=& \left(L_{\omega_1}\times L_{\omega_1}^n\right)\rtimes\langle b_1\rangle,\\
B_2&:=&L_{\omega_2}\times \left(L_{\omega_1}^n\rtimes\langle b_2\rangle\right),\\
B_i&:=&L_{\omega_i}\times L_{\omega_1}^n\times\langle b_i\rangle,\,\,\,\,\,\, \,\,\,\,\,\textrm{ for } i\in\{3,\ldots,k\},\\
C_i&:=&L_{\omega_i}\times L_{\omega_1}^{n},\,\,\,\,\,\,\,\,\,\,\,\,\,\,\,\,\,\,\,\,\,\,\,\,\,\,\,\, \textrm{ for } i\in\{1,\ldots,k\},\\
\end{eqnarray*}
where $b_1,\ldots,b_k\notin A$. For every $i\in\{1,\ldots, k\}$, there is an obvious embedding of $C_i$ in both $A$ and $B_i$. Hence, in what follows, we regard $C_i$ as a subgroup of both $A$ and $B_i$. Note that, for each $i\in \{1,\ldots,k\}$, we have $A\cap B_i=C_i$, $|B_i:C_i|=2$ and $|A:C_i|=|L:L_{\omega_i}|$.

\begin{figure}[!hhh]
  \begin{tikzpicture}[node distance =.5cm]
  \tikzset{myarrow/.style={==, thick}}
\node[circle,inner sep=0pt,label=90:$A$](A0){};
\node[below=of A0](A1){};
\node[below=of A1](A2){};
\node[below=of A2](A3){};
\node[below=of A3](A4){};
\node[below=of A4](A5){};
\node[right=of A4](B0){};
\node[right=of B0](B1){};
\node[right=of B1,circle,inner sep=0pt,label=-90:$C_1$](B2){};
\node[right=of B2](B3){};
\node[right=of B3](B4){};
\node[right=of B4,circle,inner sep=0pt,label=-90:$C_2$](B5){};
\node[right=of B5](B6){};
\node[right=of B6](B7){$\cdots$};
\node[right=of B7](B8){};
\node[right=of B8,circle,inner sep=0pt,label=-90:$C_k$](B9){};
\node[above=of B2](C0){};
\node[above=of C0](C1){};
\node[above=of C1](C2){};
\node[right=of C2,circle,inner sep=0pt,label=90:$B_1$](C3){};
\node[right=of C3](C4){};
\node[right=of C4](C5){};
\node[right=of C5,circle,inner sep=0pt,label=90:$B_2$](C6){};
\node[right=of C6](C7){};
\node[right=of C7](C8){};
\node[right=of C7](C8){$\cdots$};
\node[right=of C8](C9){};
\node[right=of C9](C10){};
\node[right=of C9,circle,inner sep=0pt,label=90:$B_k$](C11){};
\draw (A0) to (B2);
\draw(A0) to (B5);
\draw(A0) to (B9);
 \draw(B2) to node [auto,swap]{} node [below,swap] {$\,\,\,2$} (C3);
 \draw(B5) to node [auto,swap]{} node [below,swap] {$\,\,\,2$} (C6);
 \draw(B9) to node [auto,swap]{} node [below,swap] {$\,\,\,2$} (C11);
 \end{tikzpicture}
  \caption{}\label{fig1}
\end{figure}

\begin{lemma}\label{new}
The core of $C_1\cap \cdots\cap C_k$ in $A$ is $1\times L_{\omega_1}^n$.
\end{lemma}
\begin{proof}
Let $K$ be the core of $C_1\cap \cdots\cap C_k$ in $A$. Then
$$K=\bigcap_{a\in A}(C_1\cap \cdots\cap C_k)^a=\bigcap_{a\in A}\left((L_{\omega_1}\cap \cdots\cap L_{\omega_k})\times L_{\omega_1}^n\right)^a.$$
Recall that $L$ is a permutation group on $\Omega$ and that $\omega_1,\ldots,\omega_k$ are representatives of the orbits of $L$ on $\Omega$. We thus obtain that $L_{\omega_1}\cap \ldots \cap L_{\omega_k}$ is core-free in $L$ and hence $K=1\times L_{\omega_1}^n$.
\end{proof}
Let $T$ be the group given by generators and relators
$$T:=\langle A, B_1,\ldots, B_k\mid \mathcal{R}\rangle,$$
where $\mathcal{R}$ consists only of the relations in $A, B_1,\ldots, B_k$ together with the identification of $C_i$ in $A$ and $B_i$, for every $i\in\{1,\ldots,k\}$. We will obtain some basic properties of $T$ which can be deduced from any textbook on ``groups acting on graphs'', such as~\cite{DS,Wood,Serre}. 

We have adopted the notation and terminology of~\cite{DS} and will follow closely~\cite[I.4]{DS}. Using this terminology, the group $T$ is exactly the fundamental group of the graph of groups $Y$ shown in Figure~\ref{fig2}. The vertices of $Y$ are $A$, $B_1,\ldots,B_k$ and, for each $i\in\{1,\ldots,k\}$, there is a (directed) edge $C_i$ from $A$ to $B_i$.

\begin{figure}[!h]
  \begin{tikzpicture}[node distance =1.5 cm]
  \tikzset{myarrow/.style={==, thick}}
\tikzset{EdgeStyle/.append style={bend left}}
\node[circle,draw,inner sep=2pt,minimum width=1pt,label=-90:$A$](AA){};
\node[circle,draw,inner sep=2pt,minimum width=1pt,above=of AA,label=90:$B_2$](B2){};
\node[circle,draw,inner sep=2pt,minimum width=1pt,left=of B2,label=90:$B_1$](B1){};
\node[right=of B2](B){$\cdots$};
\node[circle,draw,inner sep=2pt,minimum width=1pt,right=of B,label=90:$B_k$](Bk){};
\draw(AA) to (Bk);
 \draw[->] (AA) to node [auto,swap]{} node [below,swap] {$\,\,\,C_k$} (Bk);
 \draw[->] (AA) to node [auto,swap]{} node [right,swap] {$C_2$} (B2);
 \draw[->] (AA) to node [auto,swap]{} node [below,swap] {$C_1\,\,\,$} (B1);
  \end{tikzpicture}
  \caption{}\label{fig2}
\end{figure}

It follows from~\cite[I.4.6]{DS} that the images of $A,B_1,\ldots,B_k,C_1,\ldots,C_k$ in $T$ are isomorphic to $A,B_1,\ldots,B_k,C_1,\ldots,C_k$, respectively. This allows us to identify $A,B_1,\ldots,B_k,C_1,\ldots,C_k$ with their isomorphic images in $T$ in what follows. In particular, for each $i\in \{1,\ldots,k\}$ we still have the equalities $A\cap B_i=C_i$, $|B_i:C_i|=2$ and $|A:C_i|=|L:L_{\omega_i}|$ in $T$. Let $\mathcal{T}$ be the graph  with vertex-set
$$V\mathcal{T}=T/A~\sqcup~T/B_1~\sqcup~\cdots~\sqcup~T/B_k,$$
(where $\sqcup$ denotes the disjoint union) and edge-set
$$E\mathcal{T}=\{\{Ax,B_ix\}\mid x\in T, i\in\{1\ldots,k\}\}.$$

\subsection{Results about the group $T$ and the graph $\mathcal{T}$}\label{sub:moron}
Clearly, the action of $T$ by right multiplication on $V\mathcal{T}$ induces a group of automorphisms of $\mathcal{T}$. Under this action, the group $T$ has exactly $k+1$ orbits on $V\mathcal{T}$, namely $T/A$, $T/B_1,\ldots, T/B_k$, and $k$ orbits on $E\mathcal{T}$ with representatives $\{A,B_1\},\ldots, \{A,B_k\}$. This induces a $(k+1)$-partition of the graph $\mathcal{T}$.

Observe that the set of neighbours of $A$ in $T/B_i$ is $\{B_ia\mid a\in A\}$. As $|A:(A\cap B_i)|=|A:C_i|=|L:L_{w_i}|$, we see that $A$ has $|L:L_{w_i}|$ neighbours in $T/B_i$. It follows that $A$ has valency $\sum_{i=1}^k|L:L_{\omega_i}|=|\Omega|$. A symmetric argument, with the roles of $A$ and $B_i$ reversed, shows that $B_i$ has valency $|B_i:C_i|=2$.  In particular, $\mathcal{T}$ is a $(2,|\Omega|)$-regular graph.

\begin{lemma}\label{Kernel}
The stabiliser of the vertex $A$ in $T$ is the subgroup $A$ and the kernel of the action on the neighbourhood of $A$ is $1\times L_{\omega_1}^n$.
\end{lemma}
\begin{proof}
The definition of $\mathcal{T}$ immediately gives that $A$ is the stabiliser in $T$ of the vertex $A$.  Moreover, the neighbourhood of $A$ is $\mathcal{T}(A)=\{B_ia\mid i\in \{1,\ldots,k\},a\in A\}$. Let $K$ be the kernel of the action of $A$ on $\mathcal{T}(A)$ and let $x\in K$. Clearly, $B_iax=B_ia$ if and only if $axa^{-1}\in B_i$, that is, $axa^{-1}\in A\cap B_i=C_i$.   It follows by Lemma~\ref{new} that  $K=1\times L_{\omega_1}^n$.
\end{proof}

One of the most important and fundamental properties of $\mathcal{T}$ is that it is a tree (see~\cite[I.4.4]{DS}). We now deduce some consequences from this pivotal result.

\begin{lemma}\label{tech}
For each $i\in \{1,\ldots,k\}$, we have $A\cap A^{b_i}=C_i$.
\end{lemma}
\begin{proof}
We argue by contradiction and assume that $A\cap A^{b_i}\neq C_i$ for some $i\in\{1,\ldots,k\}$. As $|B_i:C_i|=2$, we see that $B_i$ normalises $C_i$ and hence $C_i<A\cap A^{b_i}$. In particular, there exist $a,a'\in A\setminus C_i$ with $a'=a^{b_i}=b_i^{-1}ab_i$. It follows that $A$, $B_i$, $Ab_i$ and $B_iab_i$ are distinct vertices of $\mathcal{T}$. Now, the definition of $\mathcal{T}$ shows that $(A,B_i,Ab_i,B_iab_i,Ab_i^{-1}ab_i=A)$ is a cycle of length $4$ in $\mathcal{T}$ (see Figure~\ref{fig3}). This contradicts the fact that $\mathcal{T}$ is a tree and concludes the proof. 
\begin{figure}[!h]
 \begin{tikzpicture}[node distance = 1.5cm]
  \tikzset{myarrow/.style={==, thick}}
  \tikzset{EdgeStyle/.append style={bend left}}
\node[circle,draw,inner sep=2pt,minimum width=2pt,label=-90:$Ab_i^{-1}ab_i{=}A\,\,\,\,\,\,\,\,\,\,$](A){};
   \node[circle,inner sep=2pt,minimum width=2pt,right=of A,draw,label=-90:$B_i$](B){};
   \node[circle,inner sep=2pt,minimum width=2pt,right=of B,draw,label=-90:$Ab_i$](C){};
   \node[circle,inner sep=2pt,minimum width=2pt,right=of C,draw,label=-90:$B_iab_i$](D){};
  \draw (A) to (B);
  \draw (B) to (C);
  \draw (C) to (D);
  \draw[EdgeStyle] (A) to node[above,swap]{} (D);
  \end{tikzpicture}
  \caption{}\label{fig3}
\end{figure}
\end{proof}

\begin{lemma}\label{Gfaithful}
The subgroup $A$ is core-free in $T$. In particular, the group $T$ acts faithfully on $T/A$.
\end{lemma}
\begin{proof}
Let $N$ be the core of $A$ in $T$.  From Lemma~\ref{tech}, we obtain $N\leq C_1\cap\cdots\cap C_k$, and it follows from Lemma~\ref{new} that $N\leq 1\times L_{\omega_1}^n$. By construction, the group $\langle b_1,b_2\rangle$ induces a transitive permutation group on the $n+1$ coordinates of $L_{\omega_1}\times L_{\omega_1}^{n}$. As the first coordinate of the elements of $N$ is $1$ and as $N$ is invariant under $\langle b_1,b_2\rangle$, we see that every coordinate of $N$ must be equal to $1$, that is, $N=1$. The lemma now follows. 
\end{proof}

As every vertex of $\mathcal{T}$ not in $T/A$ has valency $2$, we see that $\mathcal{T}$ is the subdivision graph of a tree $\mathcal{T}_0$ with vertex set $T/A$ and valency $|\Omega|$. Clearly, $T$ acts transitively  and, in view of Lemma~\ref{Gfaithful}, faithfully on the vertices of $\mathcal{T}_0$. The tree $\mathcal{T}_0$ and the group $T$ are our main ingredients for the proof of Theorem~\ref{theo:main}. (The auxiliary graph $\mathcal{T}$ was  introduced mainly to make it more convenient to apply the results from~\cite{DS}.)

\begin{lemma}\label{residuegroup}
The stabiliser in $T$ of the vertex $A$ of $\mathcal{T}_0$ is the subgroup $A$ and the action induced by $A$ on its neighbourhood is permutation isomorphic to the action of $L$ on $\Omega$. 
\end{lemma}
\begin{proof}
Let $\pi:A\to L$ be the natural projection onto the first coordinate. In other words, if $a=(a_0,a_1,\ldots,a_n)\in A$, with $a_0\in L$ and with $a_1,\ldots,a_n\in L_{\omega_1}$, then $\pi(a)=a_0$. Clearly, the kernel of $\pi$ is $1\times L_{\omega_1}^n$, which by Lemma~\ref{Kernel} is also the kernel of the action of $A$ on the neighbourhood of the vertex $A$. Denote by $\mathcal{T}_0(A)$ the neighbourhood of $A$ in $\mathcal{T}_0$. The definitions of $\mathcal{T}$ and $\mathcal{T}_0$ yield  $\mathcal{T}_0(A)=\{Ab_ia\mid i\in \{1,\ldots,k\},a\in A\}$. Let $\varphi:\mathcal{T}_0(A)\to \Omega$ be the mapping $\varphi:Ab_ia\mapsto \omega_i^{\pi(a)}$. We show that $\varphi$ is well-defined and injective. 

Indeed, $Ab_ia=Ab_ia'$ for some $a,a'\in A$ if and only if $Ab_ia(a')^{-1}b_i^{-1}=A$, that is, $a(a')^{-1}\in A\cap A^{b_i}$. By Lemma~\ref{tech}, $A\cap A^{b_i}=C_i$. Clearly, $a(a')^{-1}\in C_i$ if and only if $\pi(a(a')^{-1})\in L_{\omega_i}$, that is, $\omega_i^{\pi(a)}=\omega_i^{\pi(a')}$. This shows that $\varphi$ is well-defined and that it is a injective. 

Clearly, $\varphi$ is surjective and hence it is a bijection. For every $a,x\in A$ and for every $i\in \{1,\ldots,k\}$, we have $\varphi((Ab_ia)x)=(\varphi(Ab_ia))^{\pi(x)}$. As $\varphi$ is a bijection, this shows that the action of $A$ on $\mathcal{T}_0(A)$ is permutation isomorphic to the action of $L$ on $\Omega$. 
\end{proof}

Recall that a group $X$ is said to be \emph{residually finite} if there exists a family $\{X_m\}_{m\in \mathbb{N}}$ of normal subgroups of finite index in $X$ with $\bigcap_{m\in \mathbb{N}}X_m=1$.

\begin{lemma}\label{residuallyfinite}
The group $T$ is residually finite.
\end{lemma}
\begin{proof}
As the groups $A$, $B_1,\ldots,B_k$ are finite, it follows from~\cite[I.4.7]{DS} that there exists a finite group $F$ and a group homomorphism $\pi:T\to F$ with $\Ker\pi\cap A=1$ and $\Ker\pi\cap B_i =1$ for each $i\in \{1,\ldots,k\}$. Write $K=\Ker\pi$. Since $F$ is finite, we have $|T:K|<\infty$. 

Since $K\unlhd T$, $K\cap A=1$ and $K\cap B_i=1$, it follows that the only element of $K$ fixing a vertex of $\mathcal{T}$ is $1$ and hence, by~\cite[I.5.4]{DS}, $K$ is a free group. In particular, $K$ is residually finite (see~\cite[6.1.9]{Robinson} for example). It follows that there exists a family $\{K_m\}_{m\in\mathbb{N}}$ of normal subgroups of finite index in $K$ with $\bigcap_{m\in \mathbb{N}}K_m=1$. 

Let $T_m$ be the core of $K_m$ in $T$. As $|T:K_m|=|T:K||K:K_m|<\infty$, we see that $|T:T_m|<\infty$. Moreover, since $T_m\leq K_m$, we have $\bigcap_{m\in \mathbb{N}}T_m=1$ and the lemma follows. 
\end{proof}

\subsection{Proof of Theorem~\ref{theo:main}. }

We now recall the definition of a normal quotient of a graph. Let $\Gamma$ be a $G$-vertex-transitive graph and let $N$ be a normal subgroup of $G$. Let $v^N$ denote the $N$-orbit containing $v\in V\Gamma$. Then the \textit{normal quotient} $\Gamma/N$  is the graph whose vertices are the $N$-orbits on $V\Gamma$, with an edge between distinct vertices $v^N$ and $w^N$ if and only if there is an edge $\{v',w'\}$ of $\Gamma$ for some $v'\in v^N$ and some $w'\in w^N$. Observe that the group $G/N$ acts transitively on the graph $\Gamma/N$.

\begin{lemma}\label{final}There exists a locally-$L$ pair $(\Gamma_n,G_n)$ such that the stabiliser of a vertex of $\Gamma_n$ in $G_n$ has order $|L||L_{\omega_1}|^n$.
\end{lemma}
\begin{proof}
By Lemma~\ref{residuallyfinite}, $T$ is residually finite and hence there exists a family $\{T_m\}_{m\in\mathbb{N}}$ of normal subgroups of finite index in $T$ with $\bigcap_{m\in \mathbb{N}}T_m=1$. Consider the set $$X=\{a_1b_i^{-1}a_2b_ja_3\mid a_1,a_2,a_3\in A,\, i,j\in \{1,\ldots,k\}\}.$$ Observe that since $A$ is finite, so is $X$. In particular, as $\bigcap_{m\in \mathbb{N}}T_m=1$ and $1\in X$, there exists $m\in\mathbb{N}$ with $X\cap T_{m}=1$. Let $G_n=T/T_{m}$ and $\Gamma_n=\mathcal{T}_0/T_m$. As $|T:T_m|<\infty$, the group $G_n$ and the graph $\Gamma_n$ are finite. Note that $\Gamma_n$ is connected and $G_n$-vertex-transitive. We first show that $\Gamma_n$ has valency $|\Omega|$. 

We argue by contradiction and suppose that $\Gamma_n$ has valency less than $|\Omega|$. It follows from the definition of normal quotient that the vertex $A$ of $\mathcal{T}_0$ must have two distinct neighbours in the same $T_m$-orbit. Recall that the neighbourhood of $A$ in $\mathcal{T}_0$ is $\{Ab_ia\mid i\in \{1,\ldots,k\},a\in A\}$. In particular, $Ab_ia\neq Ab_{j}a'$ and $Ab_ian=Ab_{j}a'$, for some $i,j\in \{1,\ldots,k\}$, $a,a'\in A$ and $n\in T_m$. It follows that $n\in a^{-1}b_i^{-1}Ab_ja'\subseteq X$ and hence $n\ \in X\cap T_{m}=1$, which is a contradiction.

Let $K$ be the kernel of the action of $G_n$ on $V\Gamma_n$. Since the valency of $\Gamma_n$ equals the valency of $\mathcal{T}_0$, we have that $\Gamma_n$ is a regular cover of $\mathcal{T}_0$. Since $\Gamma_n$ is connected, it follows that $K$ acts semiregularly on $V\Gamma_n$ and hence $K=T_m$. By Lemma~\ref{residuegroup}, $(\mathcal{T}_0,T)$ is locally-$L$ and hence so is $(\Gamma_n,G_n)$. Finally, the stabiliser of the vertex $AT_m$ of $\Gamma_n$ is $AT_m/T_m\cong A/(A\cap T_m)\cong A$, which has order $|A|=|L||L_{\omega_1}|^n$. 
\end{proof}

\begin{proof}[Proof of Theorem~\ref{theo:main}]
By Lemma~\ref{final}, for every natural integer $n\geq 2$, there exists a locally-$L$ pair $(\Gamma_n,G_n)$ with $|(G_n)_v|=|L||L_{\omega_1}|^n$, for $v\in V\Gamma_n$. As $|L_{\omega_1}|>1$, this shows that $L$ is not graph-restrictive.
\end{proof}


\thebibliography{99}

\bibitem{DS}D.~Goldschmidt, Graphs and Groups, in A.~Delgado, D.~Goldschmidt, B.~Stellmacher, \textit{Groups and graphs: new results and methods}, DMV Seminar~6, Birkh\"{a}user Verlag, Basel, 1985.

\bibitem{Wood}W.~Dicks, M.~J.~Dunwoody, \textit{Groups actings on graphs}, Cambridge studies in advanced mathematics \textbf{17}, Cambridge University Press, Cambridge, 1989.


\bibitem{PSVRestrictive} P.~Poto\v{c}nik, P.~Spiga, G.~Verret, On graph-restrictive permutation groups, \textit{J. Comb. Theory, Ser. B} \textbf{102} (2012), 820--831.

\bibitem{Robinson} D.~J.~S.~Robinson, \textit{A Course in the Theory of Groups}, Springer-Verlag, New York, 1980.

\bibitem{Serre}J.~P.~Serre, \textit{Trees}, Springer-Verlag, New York, 1980.

\bibitem{Verret} G.~Verret, On the order of arc-stabilizers in arc-transitive graphs, \textit{Bull.\ Aust.\ Math.\ Soc.} \textbf{80}  (2009),  498--505.
\end{document}